\documentclass[reqno]{amsart}
\usepackage{latexsym,amsmath,amssymb,amscd}
\usepackage{enumerate}
\usepackage[all,cmtip]{xy}

\theoremstyle{plain}

\newtheorem*{theorem*}{Theorem}
\newtheorem{theorem}{Theorem}
\newtheorem{corollary}[theorem]{Corollary}
\newtheorem{definition}[theorem]{Definition}
\newtheorem{example}[theorem]{Example}
\newtheorem{lemma}[theorem]{Lemma}

\newtheorem{proposition}[theorem]{Proposition}
\newtheorem{remark}[theorem]{Remark}
\newtheorem{setting}[theorem]{Setting}

\numberwithin{theorem}{section}
\numberwithin{equation}{section}

\renewcommand{\1}{{\bf 1}}

\newcommand{\R}{$R$}

\newcommand{\Ad}{{\rm Ad}}
\newcommand{\ad}{{\rm ad}}

\newcommand{\Aut}{{\rm Aut}}

\newcommand{\de}{{\rm d}}
\newcommand{\ee}{{\rm e}}
\newcommand{\End}{{\rm End}}

\newcommand{\GL}{{\rm GL}}
\newcommand{\Hom}{{\rm Hom}}

\newcommand{\ie}{{\rm i}}
\renewcommand{\Im}{{\rm Im}}
\newcommand{\Ind}{{\rm Ind}}

\newcommand{\Ker}{\mathrm{Ker}\,}

\newcommand{\Prim}{{\rm Prim}}

\newcommand{\SU}{{\rm SU}}

\newcommand{\spec}{\mathrm{spec}\,}
\newcommand{\U}{\mathrm{U}}

\newcommand{\id}{{\rm id}}

\newcommand{\Res}{{\rm Res}}

\newcommand{\supp}{\mathrm{supp}\,}
\newcommand{\Sp}{\mathrm{Sp}\,}

\newcommand{\Cl}{\mathrm{Cl}\,}

\newcommand{\CC}{{\mathbb C}}

\newcommand{\RR}{{\mathbb R}}
\newcommand{\TT}{{\mathbb T}}

\newcommand{\Ac}{{\mathcal A}}
\newcommand{\Bc}{{\mathcal B}}

\newcommand{\Fc}{{\mathcal F}}
\newcommand{\Hc}{{\mathcal H}}
\newcommand{\Jc}{{\mathcal J}}
\newcommand{\Kc}{{\mathcal K}}
\newcommand{\Lc}{{\mathcal L}}

\newcommand{\Pc}{{\mathcal P}}
\newcommand{\Sc}{{\mathcal S}}
\newcommand{\Tc}{{\mathcal T}}
\newcommand{\Uc}{{\mathcal U}}
\newcommand{\Vc}{{\mathcal V}}

\newcommand{\ag}{{\mathfrak a}}

\renewcommand{\gg}{{\mathfrak g}}

\newcommand{\kg}{{\mathfrak k}}

\renewcommand{\ng}{{\mathfrak n}}

\newcommand{\zg}{{\mathfrak z}}

\newcommand{\wtilde}[1]{\widetilde{#1}}
\newcommand{\what}[1]{\widehat{#1}}

\title[Groups with $T_1$ primitive ideal spaces]{AF-embeddability for  Lie groups \\with $T_1$ primitive ideal spaces}

\author{Ingrid Belti\c t\u a}
\author{Daniel Belti\c t\u a}
\address{Institute of Mathematics ``Simion Stoilow'' of the Romanian Academy,
P.O. Box 1-764, Bucharest, Romania}

\email{Ingrid.Beltita@imar.ro, ingrid.beltita@gmail.com}
\email{Daniel.Beltita@imar.ro, beltita@gmail.com}

\begin{document}
\parskip5pt

\begin{abstract} 
We study simply connected Lie groups $G$ for which the hull-kernel topology of the primitive ideal space $\Prim(G)$ of the group $C^*$-algebra $C^*(G)$ is $T_1$, that is, the finite subsets of $\Prim(G)$ are closed. 
Thus, we prove that $C^*(G)$ is AF-embeddable. 
To this end, we show that if $G$ is solvable and its action on the centre of $[G, G]$ has at least one imaginary weight, 
then $\Prim(G)$ has no nonempty quasi-compact open subsets. 
We prove in addition that connected locally compact groups with $T_1$ ideal spaces are strongly quasi-diagonal. 
\end{abstract}

\maketitle

\section{Introduction}

Considerable attention has been recently paid to 
finite-dimensional approximation properties of $C^*$-algebras, such as 
quasi-diagonality or embeddability into a $C^*$-algebra that is an inductive limit of finite-dimensional $C^*$-algebras  ---for short,  AF-embed\-dability.

In the present paper we study these properties for $C^*$-algebras that occur in the noncommutative harmonic analysis, that is, the $C^*$-algebras of connected Lie groups. 
In sharp contrast to the $C^*$-algebras of countable discrete groups, 
we have already found in \cite{BB18} that there exist solvable (hence amenable) connected Lie groups whose corresponding $C^*$-algebras are not quasi-diagonal, and in particular they are not AF-em\-beddable. 
It turns out however that affirmative results can be obtained under natural topological assumptions on the primitive ideal space, as described below. 

The primitive ideal space $\Prim(G)$ of a locally compact group $G$ is the primitive ideal space of the group $C^*$-algebra $C^*(G)$, endowed with its hull-kernel topology. 
One of the main results of this paper is: 
If $G$ is a simply connected Lie group for which $\Prim(G)$ is~$T_1$, then $C^*(G)$ is AF-embeddable 
(Theorem~\ref{AF-Lie}). 
We recall that the class of connected Lie groups $G$ for which $\Prim(G)$ is~$T_1$ was characterized in \cite{MoRo76} and \cite{Pu73} in Lie algebraic terms. 
Using that characterization, it turns out that it 
simultaneously contains 
many differing classes of solvable Lie groups that are not necessarily of type I, such as the celebrated Mautner groups,  
Dixmier groups,  
the diamond groups, 
the rigid motion groups, nilpotent Lie groups and compact groups.

If the topological condition~$T_1$ is not required, then $C^*(G)$ is still AF-embeddable  whenever
$G$ is a connected and simply connected solvable Lie group and its action on the centre of the commutator group $[G,G]$ (which is nilpotent) has a purely imaginary weight (Corollary~\ref{AF-solvable-extra}).
We recall that, without the weight conditions,  connected, simply connected and solvable Lie groups 
may not be AF-embeddable (see \cite[Th. 2.15]{BB18}).

Our approach to the above AF-embeddability result is based on a preliminary investigation of the dual topology of the solvable Lie groups, 
a topic that, despite some remarkable works, 
remains notoriously difficult even for some type I solvable Lie groups.
 Our key technical result shows that if $G$ is a second countable locally compact group
which has a closed normal subgroup $L$ such that 
$A:=G/L$ is abelian and the closure of every orbit of $G$ in $\what{L}$  is the spectrum  of the restriction 
of a unitary representation of $G$, 
then there is a homeomorphism of quasi-orbit spaces 
$$(\Prim(G)/\widehat{A})^\sim\simeq (\Prim(L)/G)^\sim.$$
(See Theorem~\ref{quasi-orbits}.) 
Using the results of \cite{Pu71}, we show that the above conditions on $G$ are satisfied for 
arbitrary connected and simply connected solvable Lie groups, with $L=[G,G]$. 
What the above homeomorphism facilitates for us is to check that, when the action of $G$ 
 action on the centre of $L$ has at least one purely imaginary weight, 
$\Prim(G)$ has no nonempty quasi-compact open subsets (Theorem~\ref{solvable-extra}), 
which further allows an application of a result of \cite{Ga20}. 
As an application of Theorem~\ref{AF-Lie}, we finally prove that if $\Ac$ is the $C^*$-algebra of a simply connected solvable Lie group, then the following implication holds true: if the image of $\Ac$ in every irreducible representation is AF-embeddable, then $\Ac$ itself is AF-embeddable 
(Corollary~\ref{thm-classR}). 
Apart for the case when $\Ac$ is a type~I $C^*$-algebra, very few other instances of that implication seem to be known in the literature, 
although the analogous implication for quasi-diagonality has long been known and is fairly elementary.

Our paper is organized as follows: 
In Section~\ref{section2} we establish some preliminary facts belonging to three topics: topologies of the spaces of closed subsets of a topological space, weak containment of representations, and the technique of weights applied to studying the topology of the space of quasi-orbits of a linear dynamical system. 
In Section~\ref{quasiorbits} we establish the basic homeomorphism theorem on quasi-orbit spaces mentioned above.
In Section~\ref{proofs} we then obtain a key result on absence of non-empty open quasi-compact subsets of $\Prim(G)$ if $G$ is a simply connected solvable Lie group such that the action of $G$ 
action on the centre of $[G,G]$ has at least one purely imaginary weight.
As a consequence, we obtain absence of non-empty open quasi-compact subsets of $\Prim(G)$ if $G$ is a simply connected solvable Lie group of type \R. 
Earlier results of this type can be found in \cite[\S 4.1]{DP17} in the special
case of Lie groups whose centers are nondiscrete. 
That condition, however, is too restrictive even for some simple and yet quite important examples of solvable Lie groups such as the Mautner group or its generalized versions from \cite{AuMo66}.
Section~\ref{section5} contains  our main result on AF-embeddability of $C^*(G)$ if $G$ is a simply connected Lie group for which $\Prim(G)$ is $T_1$. 
Finally, in Section~\ref{section7} we show that if $G$ is a simply connected solvable Lie group, then the $T_1$ condition on $\Prim(G)$ is equivalent to the condition that the image of $C^*(G)$ in every irreducible representation is  AF-embeddable, 
and this hence implies that $C^*(G)$ is itself AF-embeddable by the results of Section~\ref{section5}. 
In addition, we prove that every connected, locally compact group with $T_1$ primitive ideal space is strongly quasi-diagonal (Corollary~\ref{prim7}).

\section{Preliminaries }\label{section2}

We prove here some preliminary results for later use. 
For the notions and notation in $C^*$-algebras and topology, we refer the reader to \cite{Dix77},  \cite{BO08}, \cite{Wi07} and \cite{BkHa19}.

\subsection{The space of closed subsets of a topological space}
We first recall the topology defined on the set of all closed subsets of a topological space, seen as a set with 
a partial ordering given by inclusion. 
For details we refer the reader to \cite{GHKLMS03}.

\begin{definition}\label{ordtop}
	\normalfont
	Let $X$ be a topological space. 
	For $F\subseteq X$ closed, we define 
	$$ \downarrow\! F: =\{F'\mid F'\subseteq F, \; F' \; \text{closed in } X\}.$$
	We denote by $\Cl(X)$ the space of all closed subsets of $X$ endowed with its upper topology, that is, the topology for 
	which a sub-base  of closed sets consists of $X$ and sets of the form   $\downarrow\! F$, with $F\subseteq X$ closed.
	See \cite[Def. O-5.4]{GHKLMS03}.
\end{definition}

One of the main objects in our paper is the space of quasi-orbits of an action of a topological group on a topological space.
In these spaces the points are not necessarily closed, so we need to briefly
recall the notion of $T_0$-ization of a topological space. 
See \cite[Ch.~6]{Wi07} for more details.

\begin{definition}\label{T0ization}
	\normalfont
	Let $X$ be topological space, and consider the equivalence relation
	 $$ x\sim y  \quad \Longleftrightarrow \quad \overline{\{x\}} = \overline{\{y\}}.$$
	Then the  \emph{$T_0$-ization} of $X$ is  the quotient space 
	$$ X^\sim = X/\sim, $$
	with the corresponding quotient topology.
	The quotient map $q\colon X\to X/\sim$ is called the \emph{$T_0$-ization map}.
	\end{definition}

Part of the following lemma is known in some special cases (see \cite[Lemma 6.8]{Wi07}), but we give the proofs here for the sake of  completeness.

\begin{lemma}\label{quotient}
	Let $X$ be a topological space, and $G$ a topological group with a continuous action
	$G \times X\to X$, $(g, x)\mapsto g\cdot x$.
	Denote by $r\colon X\to X/G$ the canonical quotient map, by
	$q\colon X/G \to (X/G)^\sim$ the $T_0$-ization map, and by $Q\colon X \to (X/G)^\sim$, $Q:= q\circ r$  the quasi-orbit map.
	Then we have:
	\begin{enumerate}[\rm (i)]
		\item\label{quotient_item_-1} The map $r$ is continuous, open and surjective,  
		and  $r^{-1}(\overline{\{r(x)\}})=\overline{G\cdot x}$ for every  $x\in X$.
		\item\label{quotient_item_0} For every $x, y\in X$, $\overline{\{r(x)\}}= \overline{\{r(y)\}}$ if and only if 
				$\overline{G\cdot x} = \overline{G\cdot y}$.
		\item\label{quotient_item_i} $q$ and $Q$ are continuous, open and  surjective.
		\item\label{quotient_item_ii} The map $\iota\colon (X/G)^\sim \to \Cl(X)$, given by 
		$$ \iota(Q(x)) = \overline{G\cdot x}$$
		is well-defined, and a homeomorphism onto its image.
		\end{enumerate}
	 \end{lemma}

\begin{proof} 
	\eqref{quotient_item_-1}
	The maps
	$r$ is  continuous and surjective since it is a quotient map.
	The fact that $r$ is open is well-known: For every $D$ open in $X$,  
	$r^{-1} (r(D))= G\cdot D\subseteq X $ is open, hence $r(D)$ is open in the quotient topology.

We have that $G\cdot x=r^{-1}(\{r(x)\})\subseteq r^{-1}(\overline{\{r(x)\}})$, 
	and this last set is closed since~$r$ is continuous, 
	hence  $\overline{G\cdot x}\subseteq r^{-1}(\overline{\{ r(x)\}})$. 
	
	It remains to prove the converse inclusion. 
	We note that, since $r$ is an open map, the set $\overline{G\cdot x}$ 
	is $G$-invariant, 
	that is, $\overline{G\cdot x}=r^{-1}(r(\overline{ G\cdot x}))$. 
	Therefore its complement $X\setminus \overline{G\cdot x}$ is also $G$-invariant, thus 
	$r^{-1}(r(X\setminus \overline{G\cdot x}))=X\setminus \overline{G\cdot x}$.
	Since $(X\setminus \overline{G x})\cap G\cdot x=\emptyset$, it  follows then that 
	\begin{equation}\label{open_proof_eq1}
	r(x)\not\in r(X\setminus \overline{G\cdot x}).
	\end{equation} 
	On the other hand, using  that $r$ is an open mapping, 
	the subset $r(X\setminus \overline{G\cdot x})\subseteq X/G$ is open. 
	Therefore \eqref{open_proof_eq1} implies $\overline{\{r(x)\}}\cap r(X\setminus \overline{G\cdot x})=\emptyset$, 
	hence
	$$r^{-1}(\overline{\{r(x)\}})\cap r^{-1}(r(X\setminus \overline{G\cdot x}))=\emptyset.$$
	Using again that $X\setminus \overline{G\cdot x}$ is $G$-invariant, 
	it follows that
	$r^{-1}(\overline{\{r(x)\}})\cap (X\setminus \overline{G\cdot x})=\emptyset$. 
	Thus we get the inclusion $r^{-1}(\overline{\{r(x)\}})\subseteq \overline{G\cdot x}$.

	\eqref{quotient_item_0} 
 The assertion follows immediately from \eqref{quotient_item_-1}   for the implication "$\Leftarrow$"  using  that the map $r$ is surjective.
	
\eqref{quotient_item_i}
	The map
	$q$ is continuous and surjective since it is a quotient map.
	To prove that $q$ is open, it suffices to show that $q^{-1}(q(r(D)) )$ is open for every 
	$D$ open in $X$. 
	We have that 
	$$ \begin{aligned}
	q^{-1}(q(r(D)) )&= \{ r(x) \mid \exists\; y \in D\,, q(r(x))=q(r(y))\}\\
	& = \{ r(x) \mid \exists\; y \in D, \, \overline{\{r(x)\} }=\overline{\{r(y)\} }\}\\
	& = \{ r(x) \mid \exists\; y \in D, \, \overline{G\cdot x}=\overline{G\cdot y }\},
	\end{aligned}
	$$
	where in the last line we have used \eqref{quotient_item_0} .
	If there is $y \in D$ such that $\overline{G\cdot x}=\overline{G\cdot y }$, then $D\cap \overline{G\cdot x}\ne \emptyset$, hence 
	 $D\cap G\cdot x\ne \emptyset$. 
	 On the other hand, if $D\cap G\cdot x\ne \emptyset$, then for $y \in D\cap G\cdot x$ we have that $G\cdot x=G\cdot y$, hence 
	$\overline{G\cdot x}=\overline{G\cdot y }$. 
	Summing up, we have obtained that 
	$$ q^{-1}(q(r(D)))= \{ r(x)\mid D \cap G\cdot x \ne \emptyset\}= r(G\cdot D), $$
	therefore $q^{-1}(q(r(D)))$ is open. 
	
	The fact that $Q$ is a continuous, open and surjective map follows immediately. 
	
	\eqref{quotient_item_ii}
	The map $\iota$ is well-defined  and injective since $Q$ is surjective, while
	$Q(x)= Q(y)$ if and only if  $\overline{\{ r(x) \}} = \overline{\{ r(y) \}}$ and if and only if 
	$\overline{G\cdot x}=\overline{G\cdot y }$,  by  \eqref{quotient_item_0}.

	A sub-basis of open sets for the topology on $\Cl(X)$ is given by the sets of the form
	$$ \Uc (D) = \{F\in \Cl(X) \mid F\cap D \neq \emptyset\}, $$
	where $D$ are open sets in $X$.
		For $D\subseteq X$ open, 
	$$ \begin{aligned}
	\iota^{-1}(\Uc(D))& = \{Q(x) \mid \iota (Q(x)) \in \Uc(D)\}\\
		& = \{Q(x) \mid \overline{G\cdot x} \cap D \ne \emptyset \}\\
		& = \{Q(x) \mid G\cdot x \cap D \ne \emptyset \}= Q(G\cdot D), 
		\end{aligned}
	$$ 
hence $\iota^{-1}(\Uc(D))$ is open. 
It follows that $\iota$ is continuous.
On the other hand, using the arguments above, 
$$ 
\begin{aligned}
\iota (Q(D)) & =\{\overline {G\cdot x}\mid x\in D\} \\
& = \{\overline {G\cdot x}\mid \overline {G\cdot x}\cap D \ne \emptyset\}\\
& = \iota ((X/G)^\sim)\cap \Uc(D).
\end{aligned}
$$
It follows that $\iota \colon (X/G)^\sim \to \iota ((X/G)^\sim)$ is open as well, hence $\iota$  is a homeomorphism onto its image
(see \cite[p. TGI.30]{Bo71}).
\end{proof}

\begin{definition}\label{2xtild3}\normalfont
Whenever the conditions Lemma~\ref{quotient} are satisfied  we  denote  
$$(X/G)^\approx:= \iota((X/G)^\sim)\subseteq \Cl(X)$$
with the induced topology. 
Hence $(X/G)^\approx$ is a homeomorphic copy of $(X/G)^\sim$.
\end{definition}

\subsection{Upper topology and inner hull-kernel topology}\label{topologies}
Let $G$ be a second countable locally compact group $G$, with its   $C^*$-algebra $C^*(G)$.
We denote by 
$\Prim(G)$ the space of primitive ideals in $C^*(G)$ with the hull-kernel topology.
We denote by
$\what{G}$ the space of equivalence  classes $[\pi]$ of irreducible unitary representations $\pi$  of $G$ endowed with 
the hull-kernel topology, that is, 
 the inverse image topology  for the canonical 
map $\kappa\colon \what{G} \to \Prim(G)$, $\kappa([\pi])=\Ker \pi$.
We also note that, by the definition of the topology of $\what{G}$, the surjective mapping $\kappa$ is both open and closed.
Here and throughout this paper, for every 
continuous unitary 
representation $\pi$ of $G$, we denote also by $\pi$ its corresponding 
nondegenerate 
$*$-representation of $C^*(G)$. 
In particular, 
we identify $\what{G} = \what{C^*(G)}$. 

We denote by $\Tc(G)$
the set of all equivalence classes of unitary representations of $G$ in separable complex Hilbert spaces.
For any  $\Sc_1$, $\Sc_2\subseteq \Tc(G)$, we write $\Sc_1\preceq\Sc_2$ ($\Sc_1$ is \emph{weakly contained in} $\Sc_2$) if $\bigcap\limits_{\pi\in\Sc_1}\Ker\pi\supseteq \bigcap\limits_{\pi\in\Sc_2}\Ker\pi$, 
and $\Sc_1\approx\Sc_2$ ($\Sc_1$ is \emph{weakly equivalent} to $\Sc_2$) if $\Sc_1\preceq\Sc_2$ and $\Sc_2\preceq\Sc_1$.

On  $\Tc(G)$ we consider the inner hull-kernel topology, which restricted to $\what{G}$ is the hull-kernel topology. 
(See \cite{Fe62}.)
We recall that for every $*$-representation $T$ of the $C^*$-algebra $C^*(G)$ 
there exists a unique closed subset $\spec(T) \subseteq\widehat{G}$\textemdash the spectrum of the unitary equivalence class of $*$-representations $[T]$\textemdash which is weakly equivalent to~$[T]$, 
and 
$$\spec(T)=\{[\tau]\in\widehat{G}\mid \tau\preceq T \}.$$
(See \cite[Def. 3.4.6]{Dix77}.)

We note that for simplicity, for a unitary representation $T$ of $G$,  we sometimes write $T$ instead of $[T]$,  and $T\approx T_0$ when 
$\{[T]\}  \approx \{[T_0]\}$, where $T_0$ is another unitary representation of $G$.

\begin{remark}\label{upper-tops}
	\normalfont
The inner hull-kernel topology of $\Tc(G)$ is the coarsest topology for which the mapping 
	$$\spec\colon\Tc(G)\to\Cl(\widehat{G})$$
	is continuous. 
\end{remark}

\begin{lemma}\label{kappahat}
	Let $G$ be a second countable locally compact group. 
	Define  the map 
	$$
	\begin{gathered}
	\what{\kappa} \colon \Cl(\what{G})\to \Cl(\Prim(G)), \\
		 \what{\kappa}(\Fc): = \{\kappa([\pi]) \mid [\pi]\in \Fc\} \; \text{for all }  \Fc \in \Cl(\what{G}).
		\end{gathered}
	$$ 
	Then $\what{\kappa}$ is continuous.
\end{lemma}

\begin{proof} 
	The map $\widehat{\kappa}$ takes values in $\Cl(\Prim(G))$ since $\kappa\colon\widehat{G}\to\Prim(G)$ is a closed map. 
	To prove that $\widehat{\kappa}$ is continuous, it
 is enough to show that for every $\Sc \in \Cl(\Prim(G))$, the set 
	$\what{\kappa}^{-1}(\downarrow\! \Sc)$ is closed in $\Cl(\what{G})$. 
	By a simple computation we see that $\what{\kappa}^{-1}(\downarrow\! \Sc)= \downarrow\!\kappa^{-1}(\Sc)$, and 
	the assertion above follows since $\kappa^{-1}(\Sc)$ is a closed subset of $\what{G}$.
	\end{proof}

\begin{definition}\label{supp}
	\normalfont
	We keep the notation above. 
	For every $*$-representation $T$ of the $C^*$-algebra $C^*(G)$ we define the support of the unitary equivalence class 
	$[T]\in \Tc(G)$ by
	$$ \supp(T) := \what{\kappa} (\spec(T))\in \Cl(\Prim(G))).$$
	\end{definition}

\begin{remark}\label{supp_cont}
	\normalfont
	The map $\supp \colon \Tc(G) \to \Cl(\Prim(G))$ is continuous. 
	This is a consequence of the definition, Remark~\ref{upper-tops} and Lemma~\ref{kappahat}.
	\end{remark}

\begin{lemma}\label{suppspec}
Let $L$ be a locally compact group, and let  $\kappa \colon \what{L}\to  \Prim(L)$, $\pi \mapsto \Ker \pi$ be its canonical map.
	Then for  $T\in \Tc(L)$ and $A\in \Cl(\what L)$, 
	$\spec(T)=A$ if and only if $\supp(T) =\kappa (A)$.
	\end{lemma}
\begin{proof}
	This follows from the fact that, since $\kappa$ is an  closed  and open surjective map,  
$\kappa^{-1}(\kappa(A))=A$, and $\supp(T)=\kappa(\spec(T))$ for every $A\in \Cl(\what L)$.
\end{proof}

\begin{lemma}\label{equiv}
	Let $G$, $L$ be locally compact groups such that $G$ has  a continuous action 
	$G\times C^*(L)\to C^*(L)$ by automorhisms of $C^*(L)$,
	 and let
	$\kappa \colon \what{L}\to  \Prim(L)$, $\pi \mapsto \Ker \pi$ be the canonical map corresponding to $L$.
	Then the map 
	$$ \kappa^\approx \colon (\what{L}/G)^\approx \to (\Prim(L)/G)^\approx, 
	\quad \overline{G\cdot \pi}\mapsto \overline{G\cdot  \kappa(\pi)} 
		$$
	is a homeomorphism.
	\end{lemma}

\begin{proof}
	We recall the canonical identification $\widehat{C^*(L)} = \what L$, therefore
	$G$ acts continuously on $\what L$ 
	and $\Prim(L)$.
	Then for every $A\subseteq \what{L}$, $\overline{A}=\kappa^{-1}(\overline{\kappa(A)})$, since $\kappa$ is open, closed and surjective. 
	It follows  that for $A_1, A_2\subseteq \what{L}$,  $\overline{A_1}=\overline{A_2}$ if and only if
	 $\overline{\kappa(A_1)}=\overline{\kappa(A_2)}$.
	 Thus for $\pi_1, \pi_2\in \what{L}$ we have that 
	 $$ 
	 \overline{G\cdot \pi_1} = \overline{G\cdot \pi_2}
	  \Leftrightarrow 
	  \overline{\kappa(G\cdot \pi_1)}=  \overline{\kappa(G\cdot \pi_2)}
	\Leftrightarrow 
	\overline{G\cdot \kappa(\pi_1)}=  \overline{G\cdot \kappa(\pi_2)}.
	$$
	Therefore the map $\kappa^\approx$ is well-defined and injective.
Since $\kappa$ is surjective, $\kappa^\approx$ is surjective, as well. 
Hence $\kappa^\approx$ is bijective. 

The map of $\kappa^\approx$ is open and continuous  since  the diagram
$$
\xymatrix{
	\what{L}	\ar[r]^{\kappa} \ar[d]  &  \Prim(L) \ar[d] \\
	(\what{L}/G)^\approx \ar[r]^-{\kappa^\approx} & (\Prim(L)/G)^\approx
}
$$
is commutative, where the down-arrows are the maps $\what L\to (\what{L}/G)^\approx$, 
$\pi \mapsto \overline{G\cdot \pi}$,  and  $\Prim(L)\to (\Prim(L)/G)^\approx $, 
$\Pc \mapsto \overline{G\cdot \Pc}$, which are open and continuous, while $\kappa$ 
is continuous, closed and open. Hence $\kappa^\approx$ is a homeomorphism.
	\end{proof}

\subsection{Actions of Lie groups on abelian Lie groups, and the absence of quasi-compact open sets in the space of quasi-orbits}
Let $G$ be a topological group with its center $Z$. 
For every unitary irreducible representation $\pi\colon G\to\Bc(\Hc_\pi)$ there exists a character $\chi_\pi\colon Z\to\TT$ with $\pi(g)=\chi_\pi(g)\id_{\Hc_\pi}$ for every $g\in Z$. 
The character $\chi_\pi$ actually depends on the unitary equivalence class of $\pi$ rather than on $\pi$ itself, 
hence we obtain a well-defined mapping 
\begin{equation}\label{RG_def}
R^G\colon \widehat{G}\to\widehat{Z},\quad [\pi]\mapsto\chi_\pi.
\end{equation}

\begin{lemma}\label{res}
Let $G$ be a separable locally compact group with its center $Z$. 
Then the following assertions hold: 
\begin{enumerate}[{\rm(i)}]
	\item\label{res_item1} The mapping 
	$R^G\colon \widehat{G}\to\widehat{Z}$  
	is surjective and continuous.  
	\item\label{res_item3} There exists a surjective continuous  mapping 
	$\Res^G_Z\colon \Prim(G)\to\widehat{Z}$ satisfying ${\rm Res}^G_Z(\Ker\pi)=R^G([\pi])$ for every $[\pi]\in\widehat{G}$. 
	\item\label{res_item2} If $G$ is amenable, then
	$R^G\colon \widehat{G}\to\widehat{Z}$  and  $\Res^G_Z\colon \Prim(G)\to\widehat{Z}$ are open maps.  
	 \end{enumerate}
\end{lemma}

\begin{proof}
\eqref{res_item1}	
By 
\cite[Prop. 18.1.5]{Dix77} we obtain that, 
for every subset $S\subseteq \widehat{G}$, $R^G$ maps the closure of $S$ into the closure of $R^G(S)$, 
hence $R^G$ is continuous. 

To prove that $R^G$ is surjective, let $\chi\in\widehat{Z}$ be arbitrary and let $\pi_\chi\colon G\to\Bc(\Hc)$  be the unitary representation induced from~$\chi$. 
Since $G$ is separable, it follows by \cite[Th. 4.5]{Fe62} that $\pi\vert_Z$ is weakly equivalent to the orbit of $\chi$ under the natural action of $G$ on $Z$. 
However, since $Z$ is the center of $G$, that $G$-orbit of $\chi$ is the singleton~$\{\chi\}\subseteq\widehat{Z}$, 
and therefore $R^G([\pi_\chi])=\chi$. 

\eqref{res_item3}
The mapping $\Res^G_Z$ exists and is continuous by \eqref{res_item1}, 
using \cite[Lemma C.6]{Wi07} or \cite[Cor. 6.15]{Wi07}.

\eqref{res_item2}
As noted above, the $G$-orbit of any $\chi\in\widehat{Z}$ is $\{\chi\}\subseteq\widehat{Z}$ 
hence, since $\widehat{Z}$ is Hausdorff, it follows that $G$ acts minimally on~$\widehat{Z}$. 
Then, since $G$ is amenable, for any $\chi\in \widehat{G}$ and $\pi\in\widehat{G}$ one has $\pi\preceq\Ind_Z^G(\chi)$ if and only if $\chi\preceq\pi\vert_Z$ 
by \cite[Th. 3.3]{Go76}. 
Taking into account the definition of $R^G\colon \widehat{G}\to\widehat{Z}$, we then obtain
$$(R^G)^{-1}(\chi)
=\{\pi\in\widehat{G}\mid \chi\preceq\pi\vert_Z\}
=\{\pi\in\widehat{G}\mid \pi\preceq\Ind_Z^G(\chi)\}
=\spec(\Ind_Z^G(\chi)).$$
Therefore 
$$(\Res^G_Z)^{-1}(\chi)=\supp(\Ind_Z^G(\chi))=:S_{00}(\chi)$$
for arbitrary $\chi\in\widehat{Z}$. 
The mapping $S_{00}\colon \widehat{Z}\to\Cl(\Prim(G))$  defined  this way
is continuous as a a composition of continuous maps. 
For every $C\in \Cl(\Prim(G))$ 
we have
$$(\Res^G_Z)^\sharp(C):=\{\chi\in\widehat{Z}\mid (R^G)^{-1}(\chi)\subseteq C\}=S_{00}^{-1}(\downarrow\! C), $$
hence $(\Res^G_Z)^\sharp(C)\subseteq \widehat{Z}$ is a closed subset since 
$\downarrow\!C\subseteq \Cl(\Prim(G))$ is a closed subset and  
the map $S_{00}$ is continuous. 
Since  $\Res^G_Z\colon \Prim(G)\to\widehat{Z}$ is surjective, 
it then follows that $\Res^G_Z$ is an open mapping. 
Finally, $R^G=(\Res^G_Z)\circ\kappa$ is open since  $\kappa\colon\widehat{G}\to\Prim(G)$ is an open map.
\end{proof}

\begin{remark}
	\normalfont 
	Lemma~\ref{res} provides an alternative approach to \cite[Prop. 4.1]{DP17} that does not need continuous fields of $C^*$-algebras. 
	The connection with the earlier approach is that,  
	if the locally compact group $G$ is amenable, then one has the natural open continuous surjective mapping $\Res^G_Z\colon \Prim(G)\to\widehat{Z}$ by Lemma~\ref{res}, hence $C^*(G)$ has the structure of a continuous $C^*$-bundle on~$\widehat{Z}$ by \cite[Th. C.26]{Wi07}. 
\end{remark}

Before proceeding we need a definition.

\begin{definition}\label{roots}
	\normalfont
Let $K$ be a connected Lie group with its Lie algebra $\kg$,
$\Vc$ a finite-dimensional real vector space
and $\rho\colon K\to\End(\Vc)$ a representation of $K$ on $\Vc$.
By a slight abuse of notation, for every $T\in\End(\Vc)$ we denote by
again by $T$ its corresponding extension to a $\CC$-linear operator on
the complexification $\Vc_{\CC}:=\CC\otimes_{\RR}\Vc$, that is, the operator 
$\id_{\CC}\otimes T$ is denoted again by $T$. 
Then an $\RR$-linear functional $\lambda\colon\kg\to\CC$ is a called a
\emph{weight} of the representation $\rho$ if there exists $w\in
\Vc_{\CC}\setminus\{0\}$ for which $\de\rho(X)w=\lambda(X)w$ for every
$X\in\kg$.
We say that $\lambda$ is a \emph{purely imaginary weight} if
$\lambda(\kg)\subseteq\ie\RR$.
\end{definition}
For instance, if $K$ is a solvable Lie group and the spectrum of
$\de\rho(X)$ is contained in $\ie\RR$ for every $X\in\kg$, then there
exists a purely imaginary weight of $\rho$ as a direct consequence of
Lie's theorem on representations of solvable Lie algebras.
If $K$ is actually a nilpotent Lie group, then every weight of $\rho$ is
purely imaginary if and only if the spectrum of $\de\rho(X)$ is
contained in $\ie\RR$ for every $X\in\kg$, by the weight-space
decomposition of representations of nilpotent Lie algebras.

\begin{remark}\label{rem_abelian}
	\normalfont
	Let $Z$ be an abelian Lie group, connected and simply connected. 
	We may then assume that $Z$ is a finite dimensional real vector space.
	Let $K$ a locally compact group, and $\alpha\colon K \to \End(Z )$ be a continuous representation; it induces a continuous action
	$K\times \what{Z}\to \what{Z} $.
	On the other hand $\what Z$ can be identified with $ Z^*$, by $\xi \to \chi_\xi$, with $\chi_\xi(z)=\ee^{\ie \langle \xi, z\rangle}$, where
	$\langle\cdot  , \cdot \rangle\colon Z^*\times Z\to \RR$ is the duality bracket.
	With this identification, $\alpha$ induces a representation $\alpha^*\colon K  \to \End(Z^*)$, which is fact the contragredient of $\alpha$. 
	\end{remark}

\begin{lemma}\label{abelian}
	Let $Z$ be an abelian Lie group, connected and simply connected,  $K$ a connected and simply connected Lie group with its Lie algebra $\kg$, and  a continuous representation $\alpha\colon K \to \End(Z )$.
	Assume that the action $\de \alpha \colon \kg \to \End(Z)$ has at least a purely imaginary weight. 
	 Then $(Z^*/K)^\sim$ contains no non-empty open quasi-compact subsets.	
	\end{lemma}

\begin{proof}
With the notations in Remark~\ref{rem_abelian}, $Z$ is a vector space,  $\what{Z} \simeq Z^*$, and the hypothesis implies that 
the action $\de \alpha^*\colon \kg \to \End( Z^*)$ has at least a purely imaginary weight.

Let $\lambda_0\colon \kg \to\ie\RR$ be a purely imaginary weight of $\de \alpha^*$.  
By the weight space decomposition of $Z^*$ with respect to the Lie algebra representation $\de\alpha^*\colon \kg \to\End(Z^*)$, as discussed in \cite[\S 2.2]{BB19},  it follows that there exists a linear subspace $Z^*_0\subseteq Z^* $ satisfying the following conditions: 
\begin{itemize}
	\item  $\de \alpha^*(\kg) Z^* _0\subseteq Z^*_0$ and  $\dim_\RR Z^*/Z^*_0=2$ when $\lambda_0\not\equiv 0$, or
	$\dim_\RR Z^*/Z^*_0=1$ when $\lambda_0\equiv 0$.
	
	\item When  $\lambda_0\not\equiv 0$, the space $Z^*_1:= Z^*/Z^*_0$ has the structure of a $\CC$-vector space for which 
	\begin{equation}\label{eigenvalue}
	(\forall  y\in\kg )(\forall x\in Z^*) \quad \de\alpha^* (y)x\in\lambda_0(y)x+ Z_0^*.
	\end{equation}
\end{itemize}
Then there is  the commutative diagram 
$$\xymatrix{
	Z^*\ar[r]^{f} \ar[d]_q & Z^*_1 \ar[d]^{q_0} \\
	(Z^* /K)^\sim \ar[r]^{\widetilde{f}\ }& Z^*_1/\exp\lambda_0(\kg)
}$$
where
$$
\begin{gathered}
f\colon Z^*\to Z^*_1, \quad x\mapsto x+Z^*_0, \\
q\colon Z^* \to (Z^*/K)^\sim, \quad q(x):=\overline{K x}\\
q_0\colon Z^* _1\to Z^* _1/\exp\lambda_0(\kg ), \quad q_0(x+Z^*_0):=(\exp\lambda_0(\kg))x+ Z^*_0, \\
\end{gathered}
$$
and 
$$\widetilde{f}\colon (Z^*/ K)^\sim \to Z^*_1/\exp\lambda_0(\kg),\quad  
\widetilde{f}(\overline{K x}):=(\exp\lambda_0(\kg))x.$$ 
Here $\lambda_0(\kg)=\ie\RR$ when $\lambda_0 \not\equiv 0$, hence $\exp\lambda_0(\kg)=\TT$; otherwise $\lambda_0(\kg)=0$ and
$\exp\lambda_0(\kg)=\{1\}$. 
Then the mappings $q$ and $q_0$ are surjective, continuous, and open by \cite[Lemma 6.12]{Wi07}, while $f$ is clearly surjective, continuous, and open. 
It then  follows by the above commutative diagram that 
$\widetilde{f}$ is surjective, continuous, and open. 
Moreover, when $\lambda_0\not\equiv 0$,  $Z^*_1$ is a 1-dimensional complex vector space, and there is  a homeomorphism $\Psi_0\colon Z^*_1/\exp\lambda_0(\kg)\to[0,\infty)$; 
otherwise $\lambda_0(\kg)=0$ and we obtain a homeomorphism $\Psi_0\colon Z^*_1/\exp\lambda_0(\kg)\to \RR$. 
Neither $[0, \infty)$ nor $\RR$ have nonempty open quasi-compact subsets. 
Thus, we obtain a surjective, continuous and open map $\Psi_0\circ \tilde{f}$ from $(Z^*/K)^\sim$ onto 
a topological space that has no nonempty open quasi-compact subsets. 
Since the image of any open quasi-compact set through a continuous open mapping is again open and quasi-compact (\cite[Thm.~2, TG I 62]{Bo71}), there can be no nonempty open and quasi-compact subsets in 
$(Z^*/K)^\sim$. 
This concludes the proof.
\end{proof}

\section{Quasi-orbit spaces}\label{quasiorbits}

In the present section, unless otherwise mentioned, 
we keep the following notation and assumption. 

\begin{setting}\label{conds}
	Let $G$ be a fixed second countable locally compact group  and assume that the following conditions hold:
\begin{enumerate}[(1)]
	\item\label{cond1} There is a closed normal subgroup $L$ of $G$ such that $G/L$ is abelian.
		\item\label{minimal-i}  For every $\Pc \in \Prim(L) $ there is $T\in \what{G}$ such that
		$$
		\supp (T\vert_L)= \overline{G \cdot \Pc}.
		$$
\end{enumerate}
\end{setting}

\begin{remark}\label{cond2}\normalfont 
	Condition \eqref{minimal-i} in Setting~\ref{conds} is equivalent to the following condition. 
	\\
		(\ref{minimal-i}') For every $\pi \in \what{L} $ there is $T\in \what{G}$ such that
		$$
		\spec (T\vert_L)= \overline{G\cdot  \pi}.
		$$
	This is follows from Lemma~\ref{suppspec}.
\end{remark}

\begin{lemma}\label{minimal}
	Let $G$ be as above. 
	\begin{enumerate}[\rm (i)]
		\item\label{minimal-i1}  For every $T\in \what{G}$, there is $\pi \in \what{L}$ such that  
		$$ \supp (T\vert_L)=\overline{G\cdot  \Ker \pi}.$$
\item\label{ind} For $\pi_1, \pi_2\in \what{L}$,
	$\overline{G\cdot \pi_1}= \overline{G\cdot \pi_2}$ if and only if  $\overline{G\cdot\Ker \pi_1}= \overline{G\cdot \Ker\pi_2}$, and  if and only if 
the representations $\Ind_L^G(\pi_1)$ and $\Ind_L^G(\pi_2)$ are weakly equivalent.		
	\end{enumerate}
	\end{lemma}

\begin{proof}
The assertion \eqref{minimal-i1} is immediate from \cite[Thm.~2.1]{Go76}.
	
	\eqref{ind} By Lemma~\ref{equiv}, $\overline{G\cdot \pi_1}= \overline{G\cdot \pi_2}$ if and only if 
	$\overline{G\cdot \Ker \pi_1}= \overline{G\cdot \Ker \pi_2}$. 
	
	It it thus enough to show that
	$\overline{G\cdot \pi_1}= \overline{G\cdot \pi_2}$ if and only if 
	the representations $\Ind_L^G(\pi_1)$ and $\Ind_L^G(\pi_2)$ are weakly equivalent.	
	If $\pi_1\in \overline{G\cdot \pi_2}$, then $\pi_1 \preccurlyeq G\cdot \pi_2$, by \cite[Thm.~3.4.10]{Dix77}.
	Thus,  $\Ind_L^G(\pi_1) \preccurlyeq \Ind_L^G (\pi_2)$  by \cite[Thm.~4.2]{Fe62}, since 
	$\Ind_L^G (g\cdot \pi_2) =\Ind_L^G ( \pi_2)$ for every $g\in G$. (See \cite[Lemma~2.1.3]{CG90}.)
	The converse implication follows directly from  \cite[Thm.~4.5]{Fe62}.
	\end{proof}

For the next results, we use the notation and definitions in Subsection~\ref{topologies}.

\begin{lemma}\label{support}
	The map
	$$S_0\colon (\Prim(L)/G)^\approx \to \Cl(\Prim (G)),\;\;  \overline{G\cdot \Ker \pi } \mapsto \supp(\Ind_L^G(\pi))$$ is continuous.
\end{lemma}

\begin{proof}
	We first remark that $S_0$ is well-defined by Lemma~\ref{minimal}\eqref{ind}.
	Denote by $Q_0$ the map $\what L \to (\Prim(L)/G)^\approx $, $Q_0(\pi) = \overline{G \cdot \Ker \pi}$, that is continuous, surjective and open by  Lemma~\ref{quotient}\eqref{quotient_item_i}, \eqref{quotient_item_ii} and Lemma~\ref{equiv}.
	We have the commutative diagram
	$$\xymatrix{
		\what L  \ar[d]_{Q_0} \ar[r]^{\Ind_L^G}  & \Tc(G) \ar[d]^{\supp} \ar[r]^{\spec} & \Cl(\what G) \ar[dl]^{\what{\kappa}}
		\\
		(\Prim(L)/G)^\approx  \ar[r]^{ S_0} & \Cl(\Prim( G))
	},
	$$
	where the mapping $\supp$ is continuous by Remark~\ref{supp_cont}, while $\Ind_L^G$ is continuous by 
	\cite[Thm.~4.1]{Fe62-1}.
	
	If $D$  is an open subset in  $\Cl(\Prim (G))$, 
	then $(S_0\circ Q_0)^{-1}(D)$ is open in $\what L$, hence $Q_0((S_0\circ Q_0)^{-1}(D))$ is open in $(\Prim(L)/G)^\approx$.
	On the other hand, since $Q_0$ is surjective we have that $Q_0((S_0\circ Q_0)^{-1}(D))= S_0^{-1} (D)$.
	We have thus obtained that $S_0^{-1}(D)$ is open in  $(\Prim(L)/G)^\approx $ whenever $D$ is open 
	in $\Cl(\Prim (G))$, hence $S_0$ is continuous.
\end{proof}

\begin{lemma}\label{restriction}
	The map
	$$R\colon \Prim(G) \to (\Prim(L)/G)^\approx, \quad \Ker T \mapsto \supp(T\vert _L)$$
	is well defined,  continuous and surjective. 
\end{lemma}

\begin{proof} 
	From 
	Lemma~\ref{minimal}\eqref{minimal-i1} it follows that for every $T\in \what{G}$, there is $\pi \in \what{L}$ such that
	$\supp(T\vert _L)=\overline{G\cdot \Ker \pi}$. 
	Hence the map $R_0\colon \what{G} \to (\Prim(L)/G)^\approx$, $T \mapsto \supp(T\vert _L)$ is well-defined. 
	It is also continuous, since the restriction map $\what{G} \to \Tc(L)$, $\what{G}\to T\vert_L$ is continuous, 
	and $\supp\colon \Tc(L) \to \Cl(\Prim(L))$ is continuous.
	On the other hand, if $T_1$, $T_2\in \what{G}$ are such that $\Ker T_1= \Ker T_2$, then $R_0(T_1)= R_0(T_2)$, therefore 
	the map $R$ is well-defined and we have the commutative diagram 
	$$\xymatrix{
		\what{G} \ar[d]_{R_0} \ar[r]^{\kappa} &  \Prim(G)  \ar[dl]_{R}
		\\
		(\Prim(L)/G)^\approx
	}. 
	$$
	Since the canonical map $\kappa \colon\what{G} \to \Prim{(G)}$, $\kappa([T])=\Ker T$,  is continuous and open, it follows that $R$ is continuous as well.

	Surjectivity follows by the assumption \eqref{minimal-i} in Setting~\ref{conds}.
	 \end{proof}

We now introduce some notation needed in Theorem~\ref{quasi-orbits} below. 
Set $A:=G/L$ and let  $p\colon G\to A$ be the canonical quotient homomorphism. 
The character group $\widehat{A}$ has a continuous action by automorphisms of $C^*(G)$, 
$$\widehat{A}\times C^*(G)\to C^*(G),\quad (\chi,\varphi)\mapsto \chi\cdot \varphi$$
where $(\chi\cdot\varphi)(g):=\chi(p(g))\varphi(g)$ if $\chi\in\widehat{A}$, $\varphi\in L^1(G)$, and $g\in G$. 
This gives rise to continuous actions of $\widehat{A}$ on  $\Prim(G)$ 
and on $\widehat{C^*(G)}$, respectively.  
Taking into account the canonical identification $\widehat{C^*(G)}\simeq\widehat{G}$, 
the corresponding action of $\widehat{A}$ on classes of  irreducible representation of $G$  is given by
$$\widehat{A}\times\widehat{G}\to\widehat{G},\quad (\chi,T)\mapsto \chi\cdot T
=(\chi\circ p)\otimes T$$
where, for any unitary irreducible representation $T\colon G\to\Bc(\Hc)$ and any $\chi\in\widehat{A}$ one defines $\chi\cdot T\colon G\to\Bc(\Hc)$, $(\chi\cdot T)(g):=\chi(p(g))T(g)$. 
As usual, we denote by $(\Prim(G)/\widehat{A})^\sim$ and $(\widehat{G}/\widehat{A})^\sim$ the quasi-orbit spaces corresponding to the above actions of $\widehat{A}$ on $\Prim(G)$ and on $\widehat{G}$, respectively, and by $(\Prim(G)/\widehat{A})^\approx$ and $(\widehat{G}/\widehat{A})^\approx$ their homeomorphic copies in $\Cl(\Prim(G))$ and $\Cl(\what{G})$, respectively. 

\begin{theorem}\label{quasi-orbits}
	Let $G$ be second countable locally compact group.
	Assume that there exists a closed normal subgroup $L$ of $G$ such that $G/L$ is abelian, 
	and for every $\Pc \in \Prim(L) $ there is $T\in \what{G}$ such that
	$$
	\supp (T\vert_L)= \overline{G\cdot \Pc}.$$
	Then there is a homeomorphism of quasi-orbit spaces 
	\begin{equation}\label{quasi-orbits_eq1}
	 (\Prim(L)/G)^\sim\to (\Prim(G)/\widehat{A})^\sim.
	 \end{equation}
\end{theorem}

\begin{proof}
	We prove in fact that the map 
	\begin{equation}\label{quasi-orbits_eq1_extra}
	S\colon (\Prim(L)/G)^\approx\to (\Prim(G)/\widehat{A})^\approx,\quad 
	\overline{G\cdot \Ker \pi}\mapsto \supp(\Ind_L^G(\pi)).
	\end{equation}
	is a well-defined homeomorphism.

	We first show  that the image of $S$ is contained in $(\Prim(G)/\widehat{A})^\approx$. 
	To this end, using Remark~\ref{cond2}  and Lemma~\ref{suppspec},  it suffices to prove the following:
	\begin{equation}\label{quasi-orbits_proof_eq1}
	\pi\in\widehat{L},\  T\in\widehat{G},\ \spec(T\vert_L) =\overline{G\cdot \pi} \implies 
	\spec(\Ind_L^G(\pi))=\overline{\widehat{A}\cdot T}\in (\widehat{G}/\widehat{A})^\approx.
	\end{equation} 
	To prove \eqref{quasi-orbits_proof_eq1}, let $\lambda\colon A\to \Bc(L^2(A))$ be the regular representation of $A$. 
	Since $A$ is a locally compact abelian group, the Fourier transform $L^2(A)\to L^2(\widehat{A})$ is a unitary equivalence between $\lambda$ and the unitary representation by multiplication operators,
	$$\widehat{\lambda}\colon A\to \Bc(L^2(\widehat{A})), \quad 
	(\widehat{\lambda}(a)\psi)(\chi):=\chi(a)\psi(\chi).$$
	Using the direct integral decomposition 
	$\widehat{\lambda}=\int_{\widehat{A}}\chi\de\chi$
	with respect to a Haar measure on~$\widehat{A}$, we obtain a unitary equivalence 
	$$\int\limits_{\widehat{A}}\chi\cdot T\de\chi
	=\int\limits_{\widehat{A}}(\chi\circ p)\otimes T\de\chi
	\simeq \int\limits _{\widehat{A}}(\chi\circ p)\de\chi \otimes T
	=(\widehat{\lambda}\circ p)\otimes T.$$
	Denoting by $\tau\colon L\to\TT$ the trivial representation of $L$, 
	we have  $\lambda\circ p=\Ind_L^G(\tau)$, 
	hence there is a unitary equivalence $\widehat{\lambda}\circ p\simeq \Ind_L^G(\tau)$. 
	
	On the other hand, by \cite[Lemma 4.2]{Fe62}, there is a unitary equivalence 
	$$\Ind_L^G(\tau)\otimes T\simeq \Ind_L^G(\tau\otimes T\vert_L)=\Ind_L^G(T\vert_L).$$
	Since the support of the Haar measure of $\widehat{A}$ is equal to $\widehat{A}$, it then follows by \cite[Th. 3.1]{Fe62} that 
	the set $\widehat{A}\cdot T=\{\chi\cdot T\mid \chi\in\widehat{A}\}\subseteq\widehat{G}$ is weakly equivalent to the representation
	$\Ind_L^G(T\vert_L)$.
	Since $T\vert_L$ is weakly equivalent to $G\cdot \pi$, 
	we obtain~\eqref{quasi-orbits_proof_eq1}, by using \cite[Thm.~4.2]{Fe62} and the same argument as in Lemma~\ref{minimal}\eqref{ind} above.

	We now construct a continuous inverse of the mapping~\eqref{quasi-orbits_eq1}. 
	To this end we note that for arbitrary $T\in \widehat{G}$ and 
	$\chi\in\widehat{A}$, we have $(\chi\cdot T)\vert L=T\vert_L$, hence $R(\Ker(\chi\cdot T))=R(\Ker T)$. 
	Since the mapping $R\colon\Prim(G)\to(\Prim(L)/G)^\approx$ is continuous and surjective by Lemma~\ref{restriction}, while the topological space $(\Prim(L)/G)^\approx$ is $T_0$, it follows by \cite[Lemma 6.10]{Wi07} and Lemma~\ref{quotient}\eqref{quotient_item_ii} that there exists a continuous surjective mapping $R'\colon (\Prim(G)/\widehat{A})^\approx\to (\widehat{L}/G)^\approx$ for which the diagram 
	\begin{equation}\label{quasi-orbits_proof_eq2}
	\xymatrix{\Prim(G) \ar[d]_{\iota\circ Q} \ar[dr]^{R} & \\
		(\Prim(G)/\widehat{A})^\approx \ar[r]^{R'} & (\Prim(L)/G)^\approx}
	\end{equation}
	is commutative, where $Q$, $\iota$  are maps given by Lemma~\ref{quotient} for $X$ replaced by $\Prim(G)$ and $G$ replaced by $\widehat{A}$.
	 
	To see that the continuous mappings $R'$ and \eqref{quasi-orbits_eq1} are inverse to each other, 
	we recall (Lemma~\ref{minimal}\eqref{minimal-i1}) that 
	for every $T\in\widehat{G}$ there is $\pi\in\widehat{L}$ such that 
$
\supp (T\vert_L)=\overline{G\cdot \Ker\pi}.
$	
Then, by \eqref{quasi-orbits_proof_eq2}, we have
	\begin{align*}
	R'(\overline{\widehat{A}\cdot \Ker T})
	& =R'((\iota\circ Q)(\Ker T))  
	=
	R(\Ker T)\\
	& =\supp(T\vert_L) =\overline{G\cdot \Ker \pi}. 
	\end{align*}
	Hence 
	\begin{align*}
	(S\circ R')(\overline{\widehat{A}\cdot \Ker T})) 
	& =
	S(\overline{G\cdot\Ker \pi}) =\supp(\Ind_L^G(\pi))\\
	& \mathop{=}\limits^{\eqref{quasi-orbits_proof_eq1}}
	\kappa_G(\overline{\widehat{A}\cdot T})  
	=\overline{\widehat{A}\cdot \Ker T},
	\end{align*}
	where 
	the last equality holds true since the kernel mapping $\kappa_G\colon \widehat{G}\to\Prim(G)$ is continuous, closed, and equivariant with respect to the action of the automorphism group of~$C^*(G)$. 
	(See also Definition~\ref{supp}.)
	On the other hand, for $\pi \in \what{L}$ there is $T\in \what{G}$ such that $\supp( T\vert_L) =\overline{G\cdot \Ker \pi}$
	(Setting~\ref{conds}\eqref{minimal-i}). 
	Then by \eqref{quasi-orbits_proof_eq1} and the properties of $\kappa_G$, as above, we have
	\begin{align*}
	(R'\circ S)(\overline{G\cdot \Ker \pi}) 
	& = R'(\overline{\what{A}\cdot  \Ker T})\\
	& = R'((\iota\circ Q)(\Ker T))= R(\Ker T)=\supp(T\vert_L)= \overline{G\cdot \Ker \pi}.
	\end{align*}
	Thus $S\circ R'=\id$ and $ R'\circ S=\id$, hence 
$S^{-1}=R'$, which is a continuous mapping. 
This completes the proof of the fact that $S$ in \eqref{quasi-orbits} is a homeomorphism. 
\end{proof}


\section{AF-embeddable solvable Lie groups}\label{proofs}

In the present section we assume that $G$ is a connected simply connected solvable Lie group with Lie algebra $\gg$. 
Then the commutator subgroup  $L:=[G, G]$ is nilpotent, and a connected, simply connected, closed normal subgroup of $G$. 
We denote by $Z$ the center of $L$. 
The Lie algebra of $L$  is $[\gg, \gg]$, and we denote by $\zg$ the Lie algebra of $Z$.

The main result of this section is the following theorem. 
\begin{theorem}\label{solvable-extra}
	Let $G$ be a connected and simply connected solvable Lie group, $L=[G, G]$ and let be $Z$ the centre of $L$.
	Assume that the action of $G$ on $Z$ has at least one purely imaginary weight. 
	Then  $\Prim(G)$ has no non-empty open quasi-compact subsets.
\end{theorem}

Before proving this theorem we give some important consequences. 
We recall the following definition (see \cite{AuMo66}).
\begin{definition}\normalfont
	Let $G$ be a Lie group.
	Then $G$ is said to be of \emph{type \R}  provided the spectra of the operators of the adjoint action of $G$
	on its Lie algebra $\gg$ are contained in the unit circle.
	Equivalently, for every $X\in\gg$ the spectrum of the operator
	$\ad_{\gg}X\colon\gg\to\gg$ should be contained in $\ie\RR$.
\end{definition}

\begin{corollary}\label{solvable-bis}
	Let $G$ be a connected and simply connected solvable Lie group of type \R. 
	Then  $\Prim(G)$ has no non-empty open quasi-compact subsets.
\end{corollary}

\begin{proof}
	The condition that $G$ is of type $R$ implies that the action of $G$ on $Z$ has purely imaginary weights.
	Therefore the result is a consequence of Theorem~\ref{solvable-extra}.
	\end{proof}

We recall that if  $G$ is  a connected simply connected solvable Lie group,
the primitive ideals of $C^*(G)$ are maximal, or  equivalently, $\Prim(G)$ is $T_1$,  if and only if $G$ is of type  \R; see \cite[Thm.~2]{Pu73}.
Therefore the previous corollary may be re-written as follows.

\begin{corollary}\label{solvable}
	Let $G$ be a connected and simply connected solvable Lie group such that $\Prim(G)$ is $T_1$. 
	Then  $\Prim(G)$ has no non-empty open quasi-compact subsets.
\end{corollary}

\begin{corollary}\label{AF-solvable-extra}
		Let $G$ be a connected and simply connected solvable Lie group, $L=[G, G]$ and let be $Z$ the centre of $L$.	
Assume that the action of $G$ on $Z$ has at least one purely imaginary weight. 
	Then $C^*(G)$ is AF-embeddable. 
	\end{corollary}

\begin{proof}
	The result is a consequence of Theorem~\ref{solvable-extra} and \cite[Cor.~B]{Ga20}.
\end{proof}

\begin{corollary}\label{AF-solvable}
	Let $G$ be a connected and simply connected solvable Lie group such that $\Prim(G)$ is $T_1$.
	Then $C^*(G)$ is AF-embeddable. 
\end{corollary}

\begin{proof}
	The result is a consequence of  Corollary~\ref{solvable} and \cite[Cor.~B]{Ga20}.
\end{proof}

To prove Theorem~\ref{solvable-extra} we need the following technical tool. 

\begin{proposition}\label{openmap}
	Let $G$ be a connected and simply connected solvable Lie group, $L=[G, G]$ and let be $Z$ the centre of $L$.
	Then there  is a continuous, surjective and open map
	$$ \Phi\colon \Prim(G) \to (\what{Z}/G)^\sim.$$
\end{proposition}

\begin{proof}
	We show first that the group $G$ and its closed normal subgroup $L$ satisfy the conditions is Theorem~\ref{quasi-orbits}.
	We note that the Lie group $L$ is nilpotent  of type $I$ (even liminary).  
	The quotient $A= G/L$ is abelian since $L=[G, G]$, so it remains to show that for every 
	$\pi\in \what{L}$ there exists $T\in \what G$ such that 
	\begin{equation}\label{solvable-extra-1}
	\supp (T\vert_L) =\overline{G\cdot \Ker \pi}.
	\end{equation}
	Since $G$ is a connected simply connected solvable Lie group,
	for every $\pi \in \what{L}$ there is a closed normal subgroup $K$ of $G$, $L\subseteq K$, a representation 
	$\rho \in \what{K}$ with $\rho\vert_L= \pi$ such that $T_0=\Ind_K^G \rho$ is a factor representation.
	(See \cite[p.~83--83]{Pu73} and the references therein.)
	Let $T\in \what{G}$ be such that $T\approx T_0$. 
	By \cite[Thm.~4.5]{Fe62} we have that
	$$ T\vert_K \approx T_0\vert_K\approx G\cdot \rho \subseteq \what{K}.$$
	Hence 
	$$ T\vert_L = (T\vert_K)\vert_L \approx  G\cdot \rho\vert_L = G \cdot  \pi,$$
	and thus \eqref{solvable-extra-1} holds.

By Theorem~\ref{quasi-orbits} and the fact that $L$ is of type I, we obtain a homeomorphism
	$$
R'\colon (\Prim(G)/\widehat{A})^\sim \to  (\what{L}/G)^\sim.
	$$
	Denote by  $Q \colon \Prim(G)\to  (\Prim(G)/\widehat{A})^\sim$ the quasi-orbit map (see Lemma~\ref{quotient}); it is a continuous, open and surjective map. Thus, if we set $R:= R'\circ Q$ we obtain a continuous, open and surjective map
	$$R\colon \Prim(G) \to  (\what{L}/G)^\sim.$$

	On the other hand, by Lemma~\ref{res} there is a continuous, surjective and open map 
	$R^L \colon \what L \to \what Z$.
	It is easy to see that $R^L$ is also $G$-equivariant, hence it gives rise to a continuous, surjective and open mapping 
	$\overline{R^L}  \colon \what{L}/G \to \what{Z}/G$.

	Let $r_L\colon \what {L} /G\to (\what{L}/G)^\sim$ 
	be the canonical continuous mapping from $ \what {L} /G$ onto its $T_0$-ization 
	(in the sense of Definition~\ref{T0ization}),
	and let $r_Z\colon \widehat{Z}/G\to (\widehat{Z}/G)^\sim$ 
	be its analogous map for the group $Z$. 
	Then $r_Z\circ \overline{R^L}\colon \what{L}/G\to (\widehat{Z}/G)^\sim$ is continuous. 
	Hence 
	there is a continuous mapping 
	$\widetilde{R^L}\colon (\what{L}/G)^\sim \to(\widehat{Z}/G)^\sim$ 
	with $\widetilde{R^L}\circ r_L=r_Z\circ\overline{R^L}$,  by \cite[Lemma 6.10]{Wi07}.

	We thus obtain 
	the commutative diagram 
	$$\xymatrix{
		\what{L} \ar[d]_{q_G} \ar[r]^{R^L} \ar@/_3pc/[dd]_{\kappa_L} & \widehat{Z} \ar[d]^{q_Z} 
		\ar@/^3pc/[dd]^{\kappa_Z}
		\\
		\what{L}/G \ar[d]_{r_L} \ar[r]^{\ \ \overline{R^L}} & \widehat{Z}/G \ar[d]^{r_Z}	\\
		(\what{L}/G)^\sim \ar[r]^{\ \ \widetilde{R^L}} 
		& (\widehat{Z}/G)^\sim
	}
	$$
	where the surjective mappings $\kappa_L:=r_L\circ q_L$ and $\kappa_Z:=r_Z\circ q_Z$ are continuous and open by Lemma~\ref{quotient}\eqref{quotient_item_i}. 
	Thus $\kappa_Z\circ R ^L=\widetilde{R^L}\circ\kappa_G$, 
	where $\kappa_Z\circ R^L$ and $\kappa_L$ are surjective open mappings. 
	Therefore, by \cite[Ch. 3, \S 4, Props. 2--3]{Bo71} again, the continuous mapping $\widetilde{R^L}$ is open.

	The map
	$$\Phi :=  \wtilde{R^L}\circ R  \colon \Prim(G)\to (\what{Z}/G)^\sim$$
	is a continuous, open and surjective map, thus it
	satisfies the properties in the statement.
\end{proof}

\begin{proof}[Proof of Theorem~\ref{solvable-extra}.]
	The closed, normal subgroup $L$ is connected and simply connected, thus the quotient  $A:=G/L$ is an abelian Lie group, connected 
	and simply connected,
	and its Lie algebra is $\ag = \gg/[\gg, \gg]$. 
	
	The  action of $G$ on $Z$,  $\alpha \colon G \to \End(Z)$,  is trivial on $L$, therefore there is a group homomorphism 
	$\alpha_Z\colon A \to \End(Z)$ such that $\alpha_Z \circ Q = \alpha$, where $Q$ is the quotient map $Q\colon G\to A $.
	Then $\alpha_Z$ induces a natural action of $A$ on $Z^*\simeq \what G$, and $(Z^*/G)^\sim = (Z^*/A)^\sim$.

	Since the action of  $G$  on $Z$ has at least one purely imaginary weight, the action $\de \alpha_Z\colon \ag \to \End(Z)$ 
	has  at least 
	one purely imaginary weights. 
	Then the assertion in the statement follows from  Proposition~\ref{openmap} and Lemma~\ref{abelian}.
\end{proof}

\begin{example}[Generalized $ax+b$-groups]
	\normalfont Let $\Vc$ be a finite-dimensional real vector space, $D\in\End(\Vc)$, 
and $G_D:=\Vc\rtimes_{\alpha_D}\RR$ their corresponding generalized $ax+b$-group. 
Then  we claim that the following assertions are equivalent: 
\begin{enumerate}[{\rm(i)}]
\item\label{ax+b_item1} Either $\mathrm{Re}\, z>0$ for every $z\in\spec(D)$ or $\mathrm{Re}\, z<0$ for every $z\in\spec(D)$. 
\item\label{ax+b_item2} The $C^*$-algebra $C^*(G_D)$ is not quasidiagonal. 
\item\label{ax+b_item3} The $C^*$-algebra $C^*(G_D)$ is not AF-embeddable. 
\item\label{ax+b_item4} There exists a nonempty quasi-compact open subset of $\widehat{G_D}$. 
\item\label{ax+b_item5} There exists a nonempty quasi-compact open subset of $\Prim(G_D)$. 
\item\label{ax+b_item6} The set $\widehat{G_D}\setminus\Hom(G_D,\TT)$ is a  nonempty quasi-compact open subset of $\widehat{G_D}$.
\item\label{ax+b_item8} There exist nonzero self-adjoint idempotent elements of $C^*(G_D)$. 
 \end{enumerate}

\begin{proof}[Proof of claim]
\eqref{ax+b_item1}$\iff$\eqref{ax+b_item2}$\iff$\eqref{ax+b_item3}: 
See \cite[Th. 2.15]{BB18}.

\eqref{ax+b_item1}$\iff$\eqref{ax+b_item4}$\implies$\eqref{ax+b_item6}: 
See \cite[Th. 1.1]{GKT92}.

\eqref{ax+b_item6}$\implies$\eqref{ax+b_item4}: Obvious. 

\eqref{ax+b_item4}$\implies$\eqref{ax+b_item5}: 
The canonical mapping $\widehat{G_D}\to\Prim(G_D)$, $[\pi]\mapsto\Ker\pi$, is continuous and open, hence it maps any nonempty quasi-compact open subset of $\widehat{G_D}$ onto a nonempty quasi-compact open subset of $\Prim(G_D)$. 

\eqref{ax+b_item5}$\implies$\eqref{ax+b_item4}: 
Since there exists a nonempty quasi-compact open subset of $\Prim(G_D)$, 
it follows by Theorem~4.1 that ${\mathrm Re}\,z\ne 0$ for every $z\in\spec(D)$, hence $G_D$ is an exponential Lie group. 
In particular $G_D$ is type~I, hence the canonical mapping 
$\widehat{G_D}\to\Prim(G_D)$, $[\pi]\mapsto\Ker\pi$, is a homeomorphism. 
This shows that \eqref{ax+b_item4} follows by the hypothesis~\eqref{ax+b_item5}.

\eqref{ax+b_item8}$\implies$\eqref{ax+b_item4}: 
 We actually note a more general fact: 
If $\Ac$ is a $C^*$-algebra and $0\ne p=p^*=p^2\in\Ac$, 
then the set $Z_p:=\{[\pi]\in\widehat{\Ac}\mid \pi(p)\ne0\}$ is a nonempty quasi-compact open subset of $\widehat{\Ac}$. 
In fact, for every  $*$-representation $\pi\colon\Ac\to\Bc(\Hc)$, the operator $\pi(p)\in\Bc(\Hc)$ is an orthogonal projection hence the condition $\pi(p)\ne0$ is equivalent to  $\Vert\pi(p)\Vert\ge 1$. 
Hence $Z_p=\{[\pi]\in\widehat{\Ac}\mid \Vert\pi(p)\Vert>0\}$, and then 
$Z_p$ is open since the function $[\pi]\mapsto\Vert\pi(p)\Vert$ is lower semicontinuouson $\widehat{\Ac}$ by \cite[Prop.~3.3.2]{Dix77}.
On the other hand $Z_p=\{[\pi]\in\widehat{\Ac}\mid \Vert\pi(p)\Vert\ge 1\}$,  
hence $Z_p$ is quasi-compact by \cite[Prop. 3.3.7]{Dix77}. 
And finally, $Z_p\ne\emptyset$ since $p\ne0$.

\eqref{ax+b_item1}$\implies$\eqref{ax+b_item8}: 
See \cite[Lemma 2.3]{GKT92}.
\end{proof}
\end{example}

\section{AF-embeddability of simply connected Lie groups with $T_1$ primitive ideal spaces}\label{section5}

We prove here the main result of the paper, Theorem~\ref{AF-Lie}.
Along with the results of the preceding sections, its proof requires the following two lemmas.

\begin{lemma}\label{c-a}
	Let be  $K$ a compact group that acts on a topological space $X$ by a continuous action 
	$K \times X \to X$, and let $q\colon X \to X/K$ be the corresponding quotient map. 
	Then for $C\subseteq X/K$, we have that $C$ is quasi-compact if and only if $q^{-1}(C)$ is quasi-compact.
\end{lemma}

\begin{proof}
	Assume first that $q^{-1}(C)$ is quasi-compact. Since $q$ is surjective and continuous, 
	$C= q(q^{-1}(C))$ is quasi-compact.
	
	For the direct implication, let  $\{x_j\}_{j \in J}$ be any net in $q^{-1}(C)$. 
	Then $\{q(x_j)\}$ is a net in $C$.  
	Since $C$ is quasi-compact, selecting a suitable subnet, we may assume that there is $c\in X/K$ such that 
	$q(x_j) \to c$ in $X/K$. 
	Hence there exist $k_j$, $j\in J$, and $x\in q^{-1}(c)$ such that $k_j\cdot x_j \to x$. 
	The group $K$ is compact, therefore, again by selecting a suitable subnet, we may assume that there is $k \in K$ 
	such that $k_j \to k$, and thus $k_j^{-1} \to k^{-1}$ in $K$.
	By the continuity of the action we have then $x_j= k_j^{-1} \cdot (k_j \cdot x_j) \to k^{-1}\cdot x \in X$. 
	On the other hand, $q(k^{-1}\cdot  x) = q(x)=c$, hence $k^{-1}\cdot  x\in q^{-1} (c) \subseteq q^{-1}(C)$. 
	
	Thus every net in $q^{-1} (C)$ has a subnet that converges to some point in $q^{-1}(C)$, hence
	$q^{-1}(C)$ is quasi-compact.
\end{proof}

\begin{lemma}\label{KG2}
	Let $G_2$ a locally compact group, $K$ a compact group that acts continuously on $G_2$ and consider 
	$G_1 = K \ltimes G_2$.
	Assume that
	$\Prim(G_2)$ is $T_1$. 
	If $\Prim(G_2)$ has no non-empty open quasi-compact subsets, then $\Prim(G_1)$ has no 
	non-empty open quasi-compact subsets.
\end{lemma}

\begin{proof}
	We have that $C^*(G_1)= K \ltimes C^*(G_2)$ (\cite[Prop.~3.11]{Wi07}). 
	Then the restriction of ideals gives a mapping $R\colon \Prim (G_1) \to \left(\Prim(G_2)/K\right)^\sim$, which is continuous, surjective and open. (See \cite[Thm.~4.8]{GoLa89}.)
	Since the action of $K$ on $\Prim(G_2)$ is continuous, $K$ is compact, and $\Prim(G_2)$ is $T_1$ it follows by
	\cite[Cor., p. 213]{MoRo76} that the orbits of $K$ in $\Prim(G_2)$ are closed, hence
	$(\Prim(G_2)/K)^\sim =\Prim(G_2)/K$. 
	Thus we get that $R\colon \Prim (G_1) \to \Prim(G_2)/K$ is continuous, surjective and open.
	The quotient map $q\colon \Prim(G_2)\to \Prim(G_2)/K$ is continuous, open and surjective, as well. 
	
	Assume now that there is $C$ an open quasi-compact subset of $\Prim(G_1)$.
	Then $R(C)$ is quasi-compact and open in $\Prim(G_2)/K$, and by Lemma~\ref{c-a}, $q^{-1}(R(C))$ is a 
	quasi-compact and open subset of $\Prim(G_2)$. 
	By hypothesis this implies that $q^{-1} (R(C))=\emptyset$, hence $C= \emptyset$. 
	This completes the proof of the lemma.
\end{proof}

\begin{theorem}\label{AF-Lie}
	Let $G$ be a simply connected Lie group such that $\Prim(G)$ is $T_1$. 
	Then  both $C^*(G)$ and the reduced $C^*$-algebra  $C^*_r(G)$ are AF-embeddable. 
\end{theorem}

\begin{proof}
	We can write $G= S_1\times G_1$, where both $S_1$ and $G_1$ are simply connected Lie groups, $S_1$ is semisimple, and $G_1$
	has no semisimple factors. (See \cite[Proof of Thm.~2, p. 47]{Pu78}.)
	
	Since $S_1$ is liminary, $C^*(S_1)$ and $C^*_r(S_1)$ are AF-embeddable (see Remark~\ref{prim2}). 
	On the other hand,  $C^*(S_1)$ is nuclear  by \cite[Prop. 2.7.4]{BO08}, since it is liminary hence type~I. 
	Therefore we have that 
	\begin{align}
	C^*(G) &  \simeq C^*(S_1)\otimes C^*(G_1), \label{AF-Lie-1}\\
	C^*_r(G) &  \simeq C^*_r(S_1)\otimes C^*_r(G_1). \label{AF-Lie-2}
	\end{align}
	Thus it suffices to prove that $C^*(G_1)$ and $C_r^*(G_1)$  are AF-embeddable.
	
	From \eqref{AF-Lie-1}  it follows that there is a homeomorphism 
	\begin{equation}\label{AF-Lie-3}
	\Prim(G) \simeq \Prim(S_1) \times \Prim(G_1).
	\end{equation}
	It is straightforward to check that if the product of two topological spaces is $T_1$, then 
	each of them is $T_1$. 
	Hence, by \eqref{AF-Lie-3}, $\Prim(G_1)$ is $T_1$.
	
	Since $G_1$ has no semisimple factor, we then have by \cite[Prop.~3, p. 47]{Pu78} that $G_1=K \ltimes G_2$, where 
	$K$ and $G_2$ are simply connected Lie groups, 
	$K$ is compact, while $G_2$ is solvable of type \R.
	It follows that   $G_1$ is amenable (see  \cite[Prop.~11.13]{Pa88}), hence $C^*_r(G_1)= C^*(G_1)$. 
	By Theorem~\ref{solvable-bis},
	we have that either $G_2=\{1\}$ or there are no non-empty quasi-compact open subsets of $\Prim(G_2)$. 
	Using Lemma~\ref{KG2} we get that either $G_1=K$, or $\Prim(G_1)$ has no non-empty quasi-compact and open subsets. 
	Thus by \cite[Cor.~B]{Ga20} $C^*_r(G_1)= C^*(G_1)$   is AF-embeddable.
\end{proof}

\begin{example}[Generalized Mautner groups]
	\normalfont 
	Let $\Vc_1$ and $\Vc_2$ be finite-dimensional real vector spaces, 
	regarded as abelian additive groups $(\Vc_j,+)$ for $j=1,2$. 
Then for any compact abelian connected subgroup 
	 ${\mathbf T}\subseteq\GL(\Vc_2)=\Aut(\Vc_2)$ 
	 and any continuous group morphism $\beta\colon \Vc_1\to {\mathbf T}$, $v\mapsto \beta_v$,  
	 the corresponding semidirect product group $\Vc_2\rtimes_\beta\Vc_1$ 
	 is a connected simply connected solvable Lie group of type~R. 
	 In fact, for every $v_1\in\Vc_1$
the one-parameter group $\{\beta_{tv_1}=\ee^{t\de\beta(v_1)}\mid t\in\RR\}$ is bounded in $\End(\Vc_2)$, and this directly implies that all the eigenvalues of $\de\beta(v_1)$ are contained in $\ie\RR$.
On the other hand, the adjoint representation of the Lie group $\Vc_2\rtimes_\beta\Vc_1$ is given by
$(\Ad(v_2,v_1))(w_2,w_1)=(\de\beta(v_1)w_2,0)$ for all $v_j,w_j\in\Vc_j$ for $j=1,2$, hence the spectrum of $\Ad(v_2,v_1)$ is contained in $\ie\RR$.

	 On the other hand, the commutant ${\mathbf T'}:=\{g\in\GL(\Vc_2)\mid (\forall h\in {\mathbf T})\ gh=hg\}$ is a closed subgroup of $\GL(\Vc_2)$ 
	 with ${\mathbf T}\subseteq {\mathbf T'}$ since ${\mathbf T}$ is abelian. 
	 For any compact connected subgroup $K\subseteq {\mathbf T'}$ 
	  we define 
	 $$\alpha\colon K \to\GL(\Vc_2\times\Vc_1),\quad k\mapsto \alpha_k:=k\times\id_{\Vc_1}.$$
	 A simple computation shows that $\alpha(K)\subseteq\Aut(\Vc_2\rtimes_\beta\Vc_1)$ 
	 hence, we may define the semidirect product group $(\Vc_2\rtimes_\beta\Vc_1)\rtimes_\alpha K$.

	 It is often the case that $K$ can be selected to be simply connected.
	 For instance, when $\Vc_2=\Vc_{2a}\otimes_{\CC}\Vc_{2b}$ is a tensor product of complex vector spaces with $\dim_\CC\Vc_{2b}\ge 2$ and the image of $\beta$ consists of $\CC$-linear maps contained in 
	 $\End_\CC(\Vc_{2a})\otimes\id_{\Vc_{2b}}$. 
	 Then ${\mathbf T'}$ contains $\id_{\Vc_{2a}}\otimes\End_\CC(\Vc_{2b})$ hence, since $\dim_\CC\Vc_{2b}\ge 2$, we may select $K\subseteq {\mathbf T'}$ with $K$ isomorphic to the special unitary group $\SU(n)$ for suitable $n\ge 2$, and then $K$ is simply connected. 
	 In that case the group $(\Vc_2\rtimes_\beta\Vc_1)\rtimes_\alpha K$ satisfies the hypothesis of Theorem~\ref{AF-Lie}. 
\end{example}

 \begin{example}[Automorphisms of Heisenberg groups]
	\normalfont 
For $n\ge 1$ consider the Heisenberg group $H_{2n+1}=(\RR\times\CC^n,\cdot)$. 
We simultaneously regard $\CC^n$ as a complex Hilbert space with its usual scalar product $(\cdot\mid\cdot)$. 
and as a real vector space endowed with the symplectic form $\omega(\cdot,\cdot):=\Im(\cdot\mid\cdot)$  
and with its corresponding symplectic group
 $$\Sp(2n,\RR)=\{g\in \End_\RR(\CC^n)\mid (\forall v,w\in\CC^n)\quad 
 \omega(gv,gw)=\omega(v,w)\}.$$
 There is natural injective morphism of Lie groups $$\alpha\colon\Sp(2n,\RR)\to\Aut(H_{2n+1}), \quad 
 g\mapsto\alpha_g=\id_{\RR}\times g.$$
 The unitary group $\U(n)$ is a maximal compact subgroup of the simple Lie group $\Sp(2n,\RR)$,
 therefore $\SU(n)$ is a simply connected compact subgroup of $\Sp(2n,\RR)$.
 Hence the semidirect product group $H_{2n+1}\rtimes_\alpha\SU(n)$ 
 is a simply connected Lie group that satisfies the hypothesis of Theorem~\ref{AF-Lie}. 
\end{example}

\begin{example}[Free nilpotent Lie algebras]
	\normalfont 
Let $\Vc$ be a finite-dimensional real vector space. 
For any integer $r\ge 1$ let $ \ng_r(\Vc)$ be the 
 corresponding free $r$-step nilpotent Lie algebra generated by $\Vc$, and let
$N_r(\Vc)$  be the connected simply connected nilpotent Lie group whose Lie algebra is $\ng_r(\Vc)$. 
The correspondence $\Vc\mapsto\ng_r(\Vc)$ gives a functor from the category of finite-dimensional real vector spaces to the category of $r$-step nilpotent real Lie algebras. 
The action of that functor on morphisms gives a natural injective morphism of Lie groups 
$\alpha\colon \GL(\Vc)\to\Aut(N_r(\Vc))$. 
Then for any simply connected compact group $K\subseteq\GL(\Vc)$ 
one obtains the semidirect product group $N_r(\Vc)\rtimes_\alpha K$, 
which is a simply connected Lie group that satisfies the hypothesis of Theorem~\ref{AF-Lie}. 
\end{example}


\section{On primitively AF-embeddable $C^*$-algebras}\label{section7}

\begin{definition}\label{prim1}
	\normalfont
	A $C^*$-algebra $\Ac$ is called \emph{primitively AF-embeddable} if its primitive quotients $\Ac/\Pc$ for arbitrary $\Pc\in\Prim(\Ac)$ are AF-embeddable $C^*$-algebras.  
\end{definition}

\begin{remark}\label{prim2}
	\normalfont
	We recall that for a type~I, separable $C^*$-algebra $\Ac$, if $\Ac$ is primitively AF-embeddable, then it is AF-embeddable.  
	Indeed, since $\Ac$ is primitively AF-embeddable, it follows that for every irreducible $*$-representation $\pi\colon\Ac\to\Bc(\Hc)$ 
	the $C^*$-algebra $\pi(\Ac)\simeq\Ac/\Ker\pi$ is AF-embeddable, hence $\pi(\Ac)$ is stably finite by \cite[Lemma 1.3]{Sp88}. 
	Therefore $\Ac$ is residually finite by \cite[Prop. 3.2]{Sp88}, and  \cite[Th. 3.6]{Sp88}
	implies that $\Ac$ is AF-embeddable. 
\end{remark}

It is not clear to what extent the above Remark~\ref{prim2} carries over beyond the type~I $C^*$-algebras, that is, if every primitively AF-embeddable $C^*$-algebra is AF-embeddable.  
(See also \cite{DdDe11}.))
We prove however that this is the case for all $C^*$-algebras of connected, simply connected solvable Lie groups, irrespectively of whether they are type~I or not.
More specifically, we prove that if $\Ac$ is the $C^*$-algebra of a connected, simply connected solvable Lie group, then the following implications hold true: 
$$\begin{aligned}
&	\Ac\text{ is primitively AF-embeddable} \stackrel{(1)}{\iff} \Ac \text{ is strongly quasi-diagonal}\\
&\stackrel{(2)}{\iff} \Prim(\Ac)\text{ is }T_1  
\stackrel{(3)}{\implies} \Ac\text{ is AF-embeddable}
\end{aligned}
$$
Specifically, we prove the equivalence (1) in Corollary~\ref{thm-classR}, (2) is a consequence of 
the same corollary and of \cite{MoRo76}, while the implication (3) follows from Theorem~\ref{AF-Lie}.

The following fact is related to \cite[Th. 1.1]{BB18} and is applicable to Lie groups that need not be simply connected or solvable.

\begin{lemma}\label{prim4}
	Let $G$ be a connected Lie group. If $\Prim(G)$ is $T_1$, then the following assertions hold: 
	\begin{enumerate}[{\rm(i)}]
		\item\label{prim4_item1}
		The $C^*$-algebra $C^*(G)$ is primitively AF-embeddable. 
		\item\label{prim4_item2} 
		The Lie group $G$ is type I if and only if it is liminary. 
	\end{enumerate}
\end{lemma}

\begin{proof}
	\eqref{prim4_item1}
	The hypothesis that $\Prim(G)$ is $T_1$ is equivalent to the fact that every primitive ideal of $C^*(G)$ is maximal, which is further equivalent to the fact that every primitive quotient of $C^*(G)$ is a simple $C^*$-algebra. 
	Then, by \cite[Thm.~2, p.~161]{Po83}, for every $\Pc\in\Prim(G)$ we have that $C^*(G)/\Pc\simeq\Ac_\Pc\otimes\Kc(\Hc_\Pc)$, where $\Ac_\Pc$ is a simple $C^*$-algebra which is either 1-dimensional or a noncommutative torus, 
	and $\Hc_\Pc$ is a suitable Hilbert space. 
	To conclude the proof, we must show that if a noncommutative torus is simple, then it is AF-embeddable. 
	One way to obtain that conclusion is to combine the fact that if a noncommutative torus is a simple $C^*$-algebra, then it is an approximately homogeneous $C^*$-algebra by  \cite[Th. 3.8]{Ph06}, 
	and on the other hand every approximately sub-homogeneous $C^*$-algebra is AF-embeddable by \cite[Prop. 4.1]{Ro04}. 
	Alternatively, one can reason as follows in order to obtain the even stronger property that $\Ac_\Pc$ embeds into a unital simple AF-algebra: 
	We recall that if $\Ac_\Pc$ is a noncommutative torus and if $\Ac_\Pc$ is moreover a simple $C^*$-algebra, then $\Ac_\Pc$ has a tracial state $\tau$. (See for instance \cite[Th. 1.9]{Ph06}.) Since $\tau$ is a tracial state, the set $N_\tau:=\{a\in\Ac_\Pc\mid \tau(a^*a)=0\}$ is a closed 2-sided ideal of $\Ac_\Pc$, which is a simple $C^*$-algebra, hence $N_\tau=\{0\}$, that is, the tracial state $\tau$ is faithful. 
	Moreover, $\Ac_\Pc$ is separable, nuclear, and satisfies the Universal Coefficient Theorem of \cite{RS87}. 
	(See for instance the proof of \cite[Th. 3.8]{Ph06}.) On the other hand, every trace on a nuclear $C^*$-algebra is amenable by \cite[Prop. 6.3.4]{BO08}, hence $\tau$ is a faithful amenable trace on $\Ac_\Pc$. Since $\Ac_\Pc$ is nuclear, hence exact, it then follows by \cite[Th. A]{Sch18} that $\Ac_\Pc$ embeds into a simple AF-algebra. 
	We thus see   that $C^*(G)/\Pc$ is AF-embeddable for arbitrary $\Pc\in\Prim(G)$. 
	
	\eqref{prim4_item2} 
	The group $G$ is type~I if and only if every primitive quotient is type~I, 
	and then the assertion follows as a by-product of the above reasoning, since no simple noncommutative torus is type~I,
	as seen  for instance by using \cite[Lemma 4.2(ii)]{BB18} for the aforementioned tracial state~$\tau$.
\end{proof}

For the next result (Theorem~\ref{prim6}), the connected locally compact group $G$ does not have to be a Lie group.
This is not entirely surprising, due to the Lie theoretic characterization \cite[Th. 1]{MoRo76} of connected locally compact groups whose primitive ideal space is~$T_1$.

We start by proving the next lemma, which is essentially contained in \cite[proof of Th. 4 and Cor. 3, page 212]{MoRo76} at least in the special case of connected groups. 
 However we give here a more direct proof that does not use the above mentioned \cite[Th. 1]{MoRo76}.
  We recall that a topological group $G$ is called \emph{almost connected} if the quotient group $G/G_0$ is compact, where $G_0\subseteq G$ is the connected component that contains the unit element $\1\in G$. 

\begin{lemma}\label{prim5}
	Let $G$ be an almost connected, locally compact group, and denote by $\Lc(G)$ the set of all compact normal subgroups $H\subseteq G$ for which $G/H$ is a Lie group. 
	Then $\Prim(G)$ is~$T_1$ if and only if $\Prim(G/H)$ is $T_1$ for every $H\in\Lc(G)$. 
\end{lemma}

\begin{proof}
	For every $H\in\Lc(G)$ we denote by $j_H\colon G\to G/H$ its corresponding quotient map. 
	We define 
	$$\widehat{j_H}\colon \widehat{G/H}\to\widehat{G},\quad  [\sigma]\mapsto[\sigma\circ j_H]$$
	and 
	$$(j_H)^*\colon\Prim(G/H)\to\Prim(G), \quad \Ker\sigma\mapsto\Ker(\sigma\circ j_H).$$ 
	By \cite[Prop. 8.C.8]{BkHa19}, there exists a surjective $*$-morphism 
	$$(j_H)_*\colon C^*(G)\to C^*(G/H)$$
	satisfying 
	$\sigma\circ(j_H)_*=\sigma\circ j_H$ for every $[\sigma]\in\widehat{G/H}\simeq\widehat{C^*(G/H)}$, 
	where $[\sigma\circ j_H]\in\widehat{G}\simeq\widehat{C^*(G)}$. 
	Therefore 
	\begin{equation}
	\label{prim5_proof_eq1}
	(j_H)^*(\Pc)=((j_H)_*)^{-1}(\Pc)\text{ for every }\Pc\in\Prim(C^*(G/H)).
	\end{equation}
	The mapping $\widehat{j_H}$ is a homeomorphism of $\widehat{G/H}$ onto an open-closed subset of $\widehat{G}$ by \cite[Th. 5.4]{Li72} or \cite[Th. 5.2(i)]{Li72}.
	Similarly, 
	the mapping $(j_H)^*$ is a homeomorphism of $\Prim(G/H)$ onto a  closed subset of $\Prim(G)$ by
	\eqref{prim5_proof_eq1} and \cite[Prop. 3.2.1]{Dix77}. 
	
	Since the maps $\kappa_G\colon\widehat{G}\to\Prim(G)$, $[\pi]\mapsto\Ker\pi$ and $\kappa_{G/H}\colon\widehat{G/H}\to\Prim(G/H)$, $[\sigma]\mapsto\Ker\sigma$ are open, continuous, and surjective,  
	while the diagram 
	$$\xymatrix{\widehat{G/H}\ar[d]_{\kappa_{G/H}}  \ar[r]^{\widehat{j_H}} &   \widehat{G}  \ar[d]^{\kappa_{G}}\\
		\Prim(G/H  )\ar[r]^{(j_H)^*}  & \Prim(G) 
	}$$
	is commutative, we obtain that the image of the mapping $(j_H)^*$ is also open, hence is an open-closed subset of $\Prim(G)$. 
	
	Finally, one has 
	$$\widehat{G}=\bigcup\limits_{H\in\Lc(G)}\widehat{j_H}(\widehat{G/H})$$
	by \cite[Th. 5.4]{Li72}, since $G$ is almost connected, hence also 
	$$\Prim (G)=\bigcup\limits_{H\in\Lc(G)}(j_H)^*(\Prim(G/H))$$
	and then the assertion follows at once since we have already seen that every set 
	$(j_H)^*(\Prim(G/H))$ is an open-closed subset of $\Prim(G)$ for  $H\in\Lc(G)$. 
\end{proof}

\begin{theorem}\label{prim6}
	Let $G$ be 
	a connected, locally compact group. 
	If $\Prim(G)$ is~$T_1$, then $C^*(G)$ is primitively AF-embeddable. 
\end{theorem}

\begin{proof}
	We use the notation of Lemma~\ref{prim5} and its proof. 
	Let $\pi\colon G\to\Bc(\Hc)$ be an arbitrary unitary irreducible representation. 
	It follows by \cite[Th. 3.1]{Li72} that there exist $H\in\Lc(G)$ and an unitary irreducible representation $\pi_0\colon G/H\to\Bc(\Hc)$ with $\pi=\pi_0\circ j_H$. 
	Hence, when we extend $\pi$ and $\pi_0$ to irreducible $*$-representations of $C^*(G)$ and $C^*(G/H)$, respectively, 
	we have  
	$$\pi=\pi_0\circ (j_H)_*, $$  
	where  $(j_H)_*\colon C^*(G)\to C^*(G/H)$ is a surjective $*$-morphism.
	(See \cite[Prop. 8.C.8]{BkHa19}.)
	
	Since $\Prim(G)$ is~$T_1$, it follows by Lemma~\ref{prim5} that $\Prim(G/H)$ is~$T_1$. 
	On the other hand, $G$ is connected, hence $G/H$ is a connected Lie group, 
	hence $C^*(G/H)$ is primitively AF-embeddable by Lemma~\ref{prim4}\eqref{prim4_item1}. 
	In particular $\pi_0(C^*(G/H))$ is an AF-embeddable $C^*$-algebra. 
	Since $\pi(C^*(G))=\pi_0((j_H)_*(C^*(G)))=\pi_0(C^*(G/H))$, 
	it the follows that $\pi(C^*(G))$ is AF-embeddable.
	This completes the proof. 
\end{proof}

\begin{corollary}\label{prim7}
	Let $G$ be 
	a connected, locally compact group. 
	If $\Prim(G)$ is~$T_1$, then $C^*(G)$ is strongly quasi-diagonal. 
\end{corollary}

\begin{proof}
	Use Theorem~\ref{prim6} and the fact that every AF-embeddable $C^*$-algebra is quasi-diagonal. 
\end{proof}

In the case of connected and simply connected solvable Lie groups the above results go in the reverse direction, as well; we have the following corollary. 

\begin{corollary}\label{thm-classR}
	Let $G$ be a connected simply connected solvable Lie group. 
	Then the following assertions are equivalent. 
	\begin{enumerate}[\rm (i)]
		\item\label{thm-classR-i}
		$G$ is of type~\R. 
		\item\label{thm-classR-ii}
		$C^*(G)$ is primitively  AF-embeddable.
		\item\label{thm-classR-iii}
		$C^*(G)$ is strongly  quasi-diagonal.
	\end{enumerate}
\end{corollary}

\begin{proof}
	The implication \eqref{thm-classR-i} $\Rightarrow$  
	\eqref{thm-classR-ii} follows from Lemma~\ref{prim4} and \cite[Thm. 2, p.~161]{Pu73}.
	On the other hand, \eqref{thm-classR-ii} clearly implies  \eqref{thm-classR-iii}.
	
	We now prove that \eqref{thm-classR-iii} $\Rightarrow$  
	\eqref{thm-classR-i}, by showing that if $G$ is connected simply connected solvable and not of the type \R, then it cannot be strongly quasi-diagonal.
	To this end we show that there exists a closed 2-sided ideal $\Jc\subseteq C^*(G)$ such that the quotient $C^*(G)/\Jc$ is not strongly quasi-diagonal.
	
	It follows by \cite[Prop. 2.2, Ch. V]{AuMo66} 
	that if $G$ is a connected simply connected solvable Lie group and $G$ is not of type \R, then 
	there exists a connected simply connected closed normal subgroup $H\subseteq G$ 
	for which the quotient Lie group $G/H$ is isomorphic to one of the following Lie groups: 
	\begin{enumerate}
		\item $S_2:=\RR\rtimes\RR$, the connected real $ax+b$-group, defined via 
		$$\alpha\colon (\RR,+)\to\Aut(\RR,+), \quad \alpha(t)s=\ee^t s.$$ 
		\item $S_3^\sigma:=\RR^2\rtimes_{\alpha^\sigma}\RR$, defined for $\sigma\in\RR\setminus\{0\}$ via  
		$$\alpha^\sigma\colon (\RR,+)\to\Aut(\RR^2,+), \quad 
		\alpha^\sigma(t)=\ee^{\sigma t}
		\begin{pmatrix}
		\hfill \cos t & \sin t \\
		-\sin t & \cos t
		\end{pmatrix}.$$
		\item $S_4:=\RR^2\rtimes_{\beta}\RR^2$, defined via  
		$$\beta\colon (\RR^2,+)\to\Aut(\RR^2,+), \quad 
		\beta(t,s)=\ee^t
		\begin{pmatrix}
		\hfill \cos s & \sin s \\
		-\sin s & \cos s
		\end{pmatrix}.$$
	\end{enumerate}
	
	We get thus short exact sequence of amenable locally compact groups 
	$$\1\to H\to G\to G/H\to\1$$
	that
	leads to a short exact sequence of $C^*$-algebras 
	$$0\to\Jc\to C^*(G)\to C^*(G/H)\to 0, $$
	for a suitable closed 2-sided ideal $\Jc$ of $C^*(G)$. 
	Therefore, in order to show that $C^*(G)$ is not strongly quasi-diagonal,  
	it suffices to check that the $C^*$-algebra of the above groups 
	$S_2$, $S_3^\sigma$ with $\sigma\in\RR\setminus\{0\}$, and $S_4$ are not strongly quasi-diagonal, and 
	this fact was established in the proof of \cite[Thm.~1.1, (i) $\Rightarrow$ (iii)]{BB18}.  
\end{proof}

\subsection*{Acknowledgements}
We wish to thank the Referee for several useful remarks and suggestions.

\end{document}